\documentclass[11pt]{article}
\usepackage{amsmath}
  \usepackage{paralist}
  \usepackage{graphics}
  \usepackage{epsfig}
  \usepackage{float}
  \usepackage{graphicx}
  \usepackage{abstract}
 \usepackage{epstopdf}
 \usepackage{amssymb}
 \usepackage{cite}
 \usepackage{mathrsfs}
 \usepackage{amsthm}
 \usepackage[colorlinks=true]{hyperref}
\hypersetup{urlcolor=blue, citecolor=red}

\numberwithin{equation}{section}

\newtheorem{theorem}{Theorem}[section]

\newtheorem{lemma}[theorem]{Lemma}

\newtheorem{physical conclusion}{Physical Conclusion}

\newtheorem{definition}[theorem]{Definition}

\title{Global Strong Solution of a 2D  coupled Parabolic-Hyperbolic Magnetohydrodynamic System}
\author{ Ruikuan Liu\thanks{Email:liuruikuan2008@163.com. Supported by NSFC(11401479)
} \ \ \ \  Jiayan Yang\thanks{Corresponding author:jiayan\_{}1985@163.com;}
\\ \footnotesize $^{*,\dag}$Department of Mathematics, \footnotesize  Sichuan University
 Chengdu, \\  \footnotesize Sichuan 610064, P.R.China
 \\ \footnotesize $^\dag$Department of Mathematics,~Southwest Medical University  Luzhou,\\ \footnotesize
Sichuan 646000, P.R.China}

\begin{document}
\date{}
\maketitle
\begin{abstract}
The main objective of this paper is to study
the global strong solution of the parabolic-hyperbolic  incompressible magnetohydrodynamic (MHD) model in two dimensional space.
Based on  Agmon, Douglis and  Nirenberg's estimates for the stationary Stokes equation  and the Solonnikov's  theorem  of $L^p$-$L^q$-estimates for the evolution Stokes equation, it is
shown that the mixed-type  MHD equations exist a global strong solution.
\begin{center}
\textbf{\normalsize keywords}
\end{center}
Global strong solution,  Magnetohydrodynamics,  Stokes equation, $L^p$-$L^q$-estimates.
\end{abstract}

\section{Introduction}
We consider the following 2-D incompressible magnetohydrodynamic (MHD)
model, which describes the interaction between moving conductive
fluids and electromagnetic fields in \cite{LY},
\begin{eqnarray}\label{Eq}
\left\{
   \begin{aligned}
 & \frac{\partial u}{\partial t}+(u\cdot\nabla)u=\nu\Delta u-\frac{1}{\rho_0}\nabla p+\frac{\rho_e}{\rho_0}u\times \text{rot}A+f(x), &\text{in}\ \Omega\times [0,T),
 \\&\frac{\partial^2A}{\partial t^2}=\frac{1}{\epsilon_0\mu_0}\Delta A+\frac{\rho_e}{\epsilon_0}u-\nabla \Phi, & \text{in}\  \Omega\times [0,T),
 \\&  \nabla\cdot u=0,& \text{in}\  \Omega\times [0,T),
 \\& \nabla\cdot A=0,& \text{in} \ \Omega\times [0,T).
\end{aligned}
\right.
\end{eqnarray}
Here  $\Omega\subset\mathbb{R}^2$ is a bounded smooth domain, $T$ is any fixed time.  $u(x,t)$, $A(x,t)$, $p(x,t)$ are the velocity field, the magnetic potential and the pressure function, respectively.  $\Phi=\frac{\partial A_0}{\partial t}$ represents the magnetic pressure with the scalar electromagnetic potential $A_0$. The constants $\nu$, $\rho_0$, $\rho_e$, $\epsilon_0$, $\mu_0$ denote  kinetic viscosity,  mass density,  equivalent charge density,  electric permittivity and  magnetic permeability of free space.

In this paper, we focus on the system (\ref{Eq}) with the initial-boundary conditions
\begin{equation}\label{IN}
  u(0,x)=\phi(x),\ \ A(0,x)=\psi(x),\ \ A_t(0,x)=\eta(x), \quad \text{in}\ \Omega,
\end{equation}
\begin{equation}\label{BO}
  u(t,x)=0, \ \ \ A(t,x)=0,\ \ \ \text{on} \ \ \partial\Omega\times[0,T).
\end{equation}

Note that the MHD model (\ref{Eq}) is established based on the the Newton's second law and the Maxwell equations for the electromagnetic fields in \cite{LY}. In addition, the global weak solutions of the corresponding 3-D MHD model(\ref{Eq}) with (\ref{IN})-(\ref{BO}) has been  obtained
 by using the Galerkin technique and standard energy estimates in \cite{LY}. In this paper, what we are concerned is the global strong solution of the 2-D MHD model (\ref{Eq}) with the initial-boundary conditions (\ref{IN})$-$(\ref{BO}).

It is known that there have been huge mathematical studies on the existence of solutions to the N-dimension($N\geq2$) classical  MHD model established by  Chandrasekhar \cite{C}. In particular, Duvaut and Lions \cite{DL} constructed a global weak solution and the local strong solution to the 3-D classical MHD equations the initial boundary value problem, and properties of such solutions have been investigated by Sermange and Temam in \cite{ST}. Furthermore, some sufficient conditions for smoothness were presented for the weak solution to the 3-D classical MHD equations in \cite{HX} and some sufficient conditions of local regularity of suitable weak solutions to the 3-D classical MHD system for the points
belonging to a $C^3$-smooth part of the boundary  were obtained in \cite{V}. Also, the global strong solutions for  heat conducting 3-D classical magnetohydrodynamic flows  with non-negative density  were proved in \cite{ZX}.

Moreover, let's recall some known results for the 2-D classical and generalized  MHD equations.
It is noticed that the 2D classical MHD equations  admits a
unique global strong solution in \cite{DL,ST}. Furthermore, Ren, Wu, et.al \cite{RWXZ} have proved the global existence and the decay estimates of small smooth solution for the 2-D classical MHD equations without magnetic diffusion and Cao, Regmi and Wu \cite{CW} have obtained the global regularity for the 2-D classical MHD equations with mixed partial dissipation and magnetic diffusion. Besides, Regmi \cite{R} established the global weak solution for 2-D classical MHD equations with partial dissipation and vertical diffusion. There are also very interesting investigations about the existence of strong solutions to the 2-D classical and generalized  MHD equations, see \cite{HW,JNW,MA,ST,W,W2003,Y}  and references therein.

 However, it is worth pointing out that the incompressible MHD system (\ref{Eq}) is a  mixed-type differential difference equation, which is combined with the parabolic equation
(\ref{Eq})$_1$ and the hyperbolic equation (\ref{Eq})$_2$. The main challenge in obtaining global strong solution of 2-D MHD model(\ref{Eq}) with (\ref{IN})-(\ref{BO})
is the estimate
for $||u\times\text{rot}A||_{L^{\infty}(0,T;L^{2})}$ and $||(u\cdot \nabla)u||_{L^{\infty}(0,T;L^{2})}$. The difficulty is overcome by
applying the Solonnikov's theorem \cite{GH,M,S}   of $L^p-L^q$-estimates for the non-stationary Stokes  equations and  Agmon, Douglis and  Nirenberg's estimates \cite{AG,ADN,M} for the stationary Stokes equations. As is known, Solonnikov \cite{S} first gave the  proof of Maximal $L^p$-$L^q$-estimates for the Stokes equation (\ref{st}) using potential theoretic arguments. Recently,  Geissert,   Hess,  Hieber et.al \cite{GH} provided  a short proof of the corresponding Solonnikov's theorem in \cite{S}.

\vskip 3mm

The rest of this article is organized as follows. In Section 2, we introduce some elementary function spaces, a vital embedding theorem and some regularity results of both the non-stationary Stokes equations and stationary  Stokes equations. Section 3 is mainly devoted to the proof of global strong solution of (\ref{Eq})$-$(\ref{BO}).

\section{Preliminaries}
\subsection{Notations and definitions}
First, we introduce some notations and conventions used throughout this paper.

Let $\Omega\subset\mathbb{R}^2$ be a bounded sufficiently smooth domain.  Let $H^\tau(\Omega)(\tau=1,2)$ be the general Sobolev space on $\Omega$ with
the norm $||\cdot||_{H^\tau}$ and $L^2(\Omega)$ be the Hilbert space with the usual norm
$||\cdot||$. The space $H^1_0(\Omega)$  we mean that the completion of $C^\infty_0(\Omega)$ under the norm $||\cdot||_{H^1}$. If $\digamma$ is a Banach space, we denote by $L^p(0,T;\digamma)(1<p<\infty)$ the Banach space of the $\digamma$-valued functions defined
in the interval $(0,T)$ that are $L^p$-integrable.

We also consider the following spaces of divergence-free functions (see Temam \cite{T})
\begin{eqnarray*}
\begin{aligned}
  X&=\{u\in C^\infty_0(\Omega,\mathbb{R}^2)\  |\  \text{div}u=0 \ \ \text{in} \  \Omega\},
  \\Y&= \text{the closure of}\  X \  \text{in}\  L^2(\Omega, \mathbb{R}^2)
  \\&=\{u\in L^2(\Omega,\mathbb{R}^2) \ |\  \text{div}u=0 \ \ \text{in} \  \Omega\},
  \\Z&=\text{the closure of} \ X \  \text{in}\  H^1(\Omega, \mathbb{R}^2)
  \\&=\{u\in H^1_0(\Omega,\mathbb{R}^2) \ |\  \text{div}u=0 \ \ \text{in} \  \Omega\}.
\end{aligned}
\end{eqnarray*}

\begin{definition}
Suppose that $\phi,\eta\in Y$, $\psi\in Z$. For  any  $T>0$, a vector function $(u, A)$ is called a global weak solution of problem ({\ref{Eq}})$-$(\ref{BO}) on $(0, T)\times\Omega$
 if it satisfies the following conditions:
\begin{enumerate}
  \item $u\in L^2(0,T; Z)\cap L^\infty(0,T; Y),$
  \item $A\in L^\infty(0,T; Z),\  A_t\in L^\infty(0,T; Y),$
  \item For any function $v\in X$, there hold
\begin{equation*}
 \begin{aligned}
\int_{\Omega}u\cdot v\text{d}x + \int_{0}^{t}\int_{\Omega}(u\cdot\nabla)u\cdot v&+ \nu\nabla u\cdot\nabla v- \frac{\rho_e}{\rho_0}(u\times \text{rot}A)\cdot v \text{d}x\text{d}t\\
&=\int_{0}^{t}\int_{\Omega}f\cdot v\text{d}x\text{d}t
+\int_{\Omega}\phi\cdot v\text{d}x
  \end{aligned}
\end{equation*}
and
  \begin{equation*}
\int_{\Omega}\frac{\partial A}{\partial t}\cdot v\text{d}x + \int_{0}^{t}\int_{\Omega} \frac{1}{\epsilon_0\mu_0}\nabla u\cdot\nabla v
+\frac{\rho_e}{\epsilon_0}u\cdot v \text{d}x\text{d}t=\int_{\Omega}\eta v\text{d}x.
\end{equation*}
\end{enumerate}
\end{definition}

Now, we define strong solution of the problem ({\ref{Eq}})$-$(\ref{BO}).

\begin{definition}
Suppose that $\phi, \psi\in H^2(\Omega,\mathbb{R}^2)\cap Z$, $\eta\in Z$, $\Phi\in L^2(0,T;$ $ H^1_0(\Omega))$. $(u, A)$ is called a global strong solution to ({\ref{Eq}})$-$(\ref{BO}), if $(u, A)$  satisfy
 \begin{equation*}
\begin{aligned}
& u\in L^\infty_{\text{loc}}(0,\infty; H^2(\Omega,\mathbb{R}^2)\cap Z),\ u_t\in L^\infty_{\text{loc}}(0,\infty;Y)\cap L^2_{\text{loc}}(0,\infty;Z)
\\ & p\in L^\infty_{\text{loc}}(0,\infty; H^1(\Omega)),
\\ & A \in L^\infty_{\text{loc}}(0,\infty; H^2(\Omega,\mathbb{R}^2)\cap Z),\ A_t\in L^\infty_{\text{loc}}(0,\infty;Z),\ A_{tt}\in L^\infty_{\text{loc}}(0,\infty;Y).
\end{aligned}
\end{equation*}
 Furthermore, both (\ref{Eq}) and (\ref{BO})  hold almost everywhere in $\Omega\times(0,T)$.
\end{definition}

\subsection{Lemmas}
Some more lemmas will  be frequently used later. One is the following embedding result in \cite{M}.
\begin{lemma}
For any $k\geq0$, the following hold
\begin{equation}\label{in0}
L^p((0,T),W^{k+1,p}(\Omega))\cap L^\infty(0,T;L^r(\Omega))\subset L^q(0,T;W^{k,q}(\Omega)),
\end{equation}
where $q=(r(k+1)p+np)/(rk+n)$. In the special case of $k=0$, (\ref{in0}) equals to
\begin{equation}\label{in00}
L^p(0,T;W^{1,p}(\Omega))\cap L^\infty(0,T;L^r(\Omega))\subset L^q((\Omega)\times(0,T)),
\end{equation}
provided that $q=(n+r)p/n$.
\end{lemma}

\begin{proof}
From Gagliardo-Nirenberg interpolation  inequality, we have
\begin{equation}\label{in}
||u||_{W^{k,q}}\leq C||u||^\theta_{W^{m,p}}||u||^{1-\theta}_{W^{j,r}},\quad 0\leq \theta\leq1,
\end{equation}
provided that
\begin{equation*}
\theta\bigg(m-\frac{n}{p}\bigg)+(1-\theta)\bigg(j-\frac{n}{r}\bigg)\geq k-\frac{n}{q},
\end{equation*}
where $C$ is a constant independent of $u$.
\\
Inserting  $j=0,\ q\geq p$, $m=k+1$ and $\theta=\frac{p}{q}$ into (\ref{in}), it is easy to see that
\begin{equation*}
\bigg(\int_\Omega|D^ku|^q\text{d}x\bigg)^{\frac{1}{q}}\leq C\bigg(\int_\Omega|D^{k+1}u|^p\text{d}x\bigg)^{\frac{1}{q}}\bigg(
\int_{\Omega}|u|^r\text{d}x\bigg)^{\frac{1}{r}(1-p/q)},
\end{equation*}
where $q=\frac{(r(k+1)p+np)}{rk+n}$.\\ Then we get
\begin{equation*}
\int_{0}^{T}\int_\Omega|D^ku|^q\text{d}x\text{d}t\leq C\sup\limits_{0\leq t\leq T}||u||^{(q-p)r}_{L^r}\int_{0}^{T}\int_\Omega |D^{k+1}u|^p\text{d}x\text{d}t,
\end{equation*}
which implies (\ref{in0}) and (\ref{in00}).
\end{proof}

The other lemma is responsible for the estimates for $u,p,u_t$ and follows from the
$L^p$-$L^q$-estimates \cite{GH,S} for non-stationary Stokes equations. For its proof, refer to \cite{GH,S}.

Let us consider the following Stokes equations
 \begin{eqnarray}
\label{st} \left\{
   \begin{aligned}
 & \frac{\partial u}{\partial t}=\nu\Delta u-\nabla p+f(x,t),
 \\&  \nabla\cdot u=0,
 \\&  u|_{\partial \Omega}=0,
 \\&  u(0)=u_0,
\end{aligned}
\right.
\end{eqnarray}
where $\nu>0$ is a constant.

\begin{lemma}
Let $\Omega\subset\mathbb{R}^n(n=2,3)$ be a
domain with compact $C^3$-boundary, $1<r,p<\infty,\ 0<T<\infty$. Then for any
$f\in L^r(0,T;L^q(\Omega,\mathbb{R}^n))$ and $u_0\in W^{2,q}(\Omega,\mathbb{R}^n))$
there exists a unique solution $(u,p)$ of (\ref{st}) satisfying
\begin{equation*}
\begin{aligned}
& u\in L^r(0,T;W^{2,q}(\Omega,\mathbb{R}^n)),\ u_t\in L^r(0,T;L^{q}(\Omega,\mathbb{R}^n)),
\\ & p\in L^r(0,T;W^{1,q}(\Omega))
\end{aligned}
\end{equation*}
such that
\begin{equation*}
\begin{aligned}
 &||u||_{L^r(0,T;W^{2,q})}+||u_t||_{L^r(0,T;L^q)}+||p||_{L^r(0,T;W^{1,q})}
 \\&\leq
C(||f||_{L^r(0,T;L^q)}+||u_0||_{W^{2,q}}),
\end{aligned}
\end{equation*}
where $C>0$ is a constant.
\end{lemma}

Finally, we give some regularity results for the stationary Stokes system. For its proof, refer to \cite{AG,ADN,M}.

\begin{lemma}
Assume that $(v,p)\in W^{2,p}(\Omega,\mathbb{R}^n)\times W^{1,p}(\Omega)(1<p<\infty)$ is a weak solution of the stationary Stokes equations
\begin{eqnarray*}
 \left\{
   \begin{aligned}
  &-\nu\Delta v-\nabla p=F(x),     &\ \text{in}\quad \Omega,
 \\ &\nabla\cdot v=0,            & \ \text{in}\quad \Omega,
 \\ &v|_{\partial \Omega}=0,    &\ \text{on}\quad \partial\Omega,
\end{aligned}
\right.
\end{eqnarray*}
and $F\in W^{k,q}(\Omega,\mathbb{R}^n)(k\geq0,1<q<\infty)$. Then it holds that
\begin{equation*}
(v,p)\in W^{k+2,q}(\Omega,\mathbb{R}^n)\times W^{k+1,q}(\Omega)
\end{equation*}
and
\begin{equation*}
||v||_{W^{k+2,q}}+||p||_{W^{k+1,q}}\leq C(||F||_{W^{k,q}}+||(u,p)||_{L^q})
\end{equation*}
with some constant $C$ depending on $n$, $\Omega$ and $q$.
\end{lemma}

\section{Main Results}

In this section, we state the global weak solution existence theorem and the global strong solution existence one for the problem (\ref{Eq})$-$(\ref{BO}), and also prove them.

\begin{theorem}
Let the initial value $\phi,\eta\in Y$, $\psi\in Z$. If $f\in Y, \Phi\in L^2(0,T; H^1_0(\Omega))$,
then there exists a global weak solution for the problem (\ref{Eq})$-$(\ref{BO}).
\end{theorem}
\begin{proof}
By the standard Galerkin method and the similar estimates in \cite{LY},  the existence of global weak solution of (\ref{Eq})$-$(\ref{BO}) is also valid, we omit it.
\end{proof}
\begin{theorem}
Let $\Omega$ be a bounded
domain with compact $C^3$-boundary. If  $\phi, \psi\in H^2(\Omega,\mathbb{R}^2)\cap Z$, $\eta\in Z$, for any $f\in Y, \Phi\in L^2(0,T; H^1_0(\Omega))$,
then there exists a global strong solution for the problem (\ref{Eq})$-$(\ref{BO}),  i.e.,  for any $\ 0<T<\infty$
\begin{equation*}
\begin{aligned}
& u\in L^\infty(0,T; H^2(\Omega,\mathbb{R}^2)\cap Z),\ u_t\in L^\infty(0,T;Y)\cap L^2(0,T;Z)
\\ & p\in L^\infty(0,T; H^1(\Omega)),
\\ & A \in L^\infty(0,T; H^2(\Omega,\mathbb{R}^2)\cap Z),\ A_t\in L^\infty(0,T;Z),\ A_{tt}\in L^\infty(0,T;Y).
\end{aligned}
\end{equation*}

\end{theorem}

 \begin{proof}
 The proof can be divided into 3 steps. We will use the same generic constant $C$ to
 denote various constants that depend on $\mu_0,\rho_0,\rho_e,\epsilon_0$ and $T$ only.
\\
{\textbf{Step 1  The estimates and regularity for}} $A$.

From Theorem 3.1, for any $0<T<\infty$,  we get the global weak solution
\begin{equation}\label{P0}
 \begin{aligned}
 & u\in L^2(0,T; Z)\cap L^\infty(0,T; Y),
 \\ & A\in L^\infty(0,T; Z),\  A_t\in L^\infty(0,T; Y).
  \end{aligned}
\end{equation}
Multiplying both sides of (\ref{Eq})$_2$ by  $-\Delta A_t$ and
integrating over $\Omega$, we have
\begin{equation}\label{P1}
\frac{1}{2}\frac{\text{d}}{\text{d}t}\bigg(\int_\Omega|\nabla A_t|^2+\frac{1}{\epsilon_0\mu_0}|\Delta A|^2 \text{d}x\bigg)
=\frac{\rho_e}{\epsilon_0}\int_\Omega \nabla u\nabla A_t \text{d}x
\end{equation}
since $\text{div}A=0$ and (\ref{BO}).
\\ Using the H\"{o}lder inequality, it is easy to see that
\begin{equation}\label{P2}
\begin{aligned}
\frac{\text{d}}{\text{d}t}\bigg(||\nabla A_t||^2_{L^2}+\frac{1}{\epsilon_0\mu_0}||\Delta A||^2_{L^2}\bigg)\leq 2\bigg(||\nabla A_t||^2_{L^2}+\frac{1}{\epsilon_0\mu_0}||\Delta A||^2_{L^2}+\frac{\rho^2_e}{\epsilon^2_0}||\nabla u||^2_{L^2}\bigg).
\end{aligned}
\end{equation}
 Then, by the Gronwall
inequality,  (\ref{P2}) implies
\begin{equation}\label{P3}
\begin{aligned}
||\nabla A_t||^2_{L^2}+||\Delta A||^2_{L^2}\leq e^{2T}\bigg(||\Delta \psi||_{L^2}+||\nabla \eta||_{L^2}+2\frac{\rho^2_0}{\epsilon^2_0}\int_{0}^{T}||\nabla u||^2_{L^2}\text{d}s\bigg),
\end{aligned}
\end{equation}
for $\forall~0<T<\infty$.
\\
Therefore, we conclude that
\begin{equation}\label{P4}
\nabla A_t\in L^\infty(0,T;Y),\ \Delta A\in L^\infty(0,T;Y).
\end{equation}

Next, we need to derive an estimate on $||A_{tt}||_{L^\infty(0,T;Y)}$.

Multiplying both sides of Eqs. (\ref{Eq})$_2$ by  $A_{tt}$  integrating over $\Omega$ lead to \begin{equation}\label{P51}
\int_{\Omega} |A_{tt}|^2\text{d}x=\frac{1}{\epsilon_0\mu_0}\int_{\Omega}\Delta A A_{tt}\text{d}x+\frac{\rho_0}{\epsilon_0}\int_{\Omega}
u A_{tt}\text{d}x
\end{equation}
since $-\int_{\Omega}\nabla\Phi A_{tt}\text{d}x=\int_{\Omega}\Phi \text{div}A_{tt}\text{d}x=0$.\\
Using the H\"{o}der inequality and Young inequality, we deduce from (\ref{P51}) that
 \begin{equation}\label{P52}
\int_{\Omega} |A_{tt}|^2\text{d}x\leq\frac{1}{\epsilon^2_0\mu^2_0}\int_{\Omega}|\Delta A|^2\text{d}x + \frac{\rho^2_0}{\epsilon^2_0}\int_{\Omega}|u|^2\text{d}x
+\frac{1}{2}\int_{\Omega}|A_{tt}|^2\text{d}x .
\end{equation}
It is easy to see that
 \begin{equation}\label{P5}
\text{ess}\sup\limits_{0\leq t\leq T}\int_{\Omega} |A_{tt}|^2\text{d}x
\leq \sup\limits_{0\leq t\leq T}\frac{2}{\epsilon^2_0\mu^2_0}\int_{\Omega} |\Delta A|^2\text{d}x+
\sup\limits_{0\leq t\leq T}\frac{2\rho^2_0}{\epsilon^2_0}\int_{\Omega} |u|^2\text{d}x.
\end{equation}
Putting  the estimates (\ref{P0}), (\ref{P4}) and (\ref{P5}) together, we have
\begin{equation}\label{P6}
A_{tt}\in L^\infty(0,T;Y).
\end{equation}
Hence, (\ref{P4}) and (\ref{P6}) imply the regularity for $A$.
\\
{\textbf{Step 2 The} $L^{\frac{4}{3}}$-$L^{\frac{4}{3}}$-\textbf{estimates for}  $u\cdot\nabla u$ and $u\times A$}.

From (\ref{P0}) and Lemma 2.3(the case that k=0), it is easy to check that
\begin{equation}\label{P00}
u\in L^4((0,T)\times\Omega).
\end{equation}
Note that
\begin{equation}\label{P01}
 \int_{0}^{T}\int_{\Omega}|Du|^{\frac{4}{3}}|u|^{\frac{4}{3}}\text{d}x\text{d}t
 \leq \bigg(\int_{0}^{T}\int_{\Omega}|Du|^{2}\text{d}x\text{d}t\bigg)^{\frac{2}{3}}
 \bigg(\int_{0}^{T}\int_{\Omega}|u|^{4}\text{d}x\text{d}t\bigg)^{\frac{1}{3}},
\end{equation}
which implies that
\begin{equation}\label{P02}
 u\cdot\nabla u\in L^{\frac{4}{3}}(0,T;L^{\frac{4}{3}}(\Omega,\mathbb{R}^2)).
\end{equation}
 Combining (\ref{P0}) and (\ref{P00}),
 we get
 \begin{equation}\label{P03}
 \begin{aligned}
\int_{0}^{T}\int_{\Omega}|u\times\text{rot}A|^{\frac{4}{3}}\text{d}x\text{d}t
  &\leq \bigg(\int_{0}^{T}\int_{\Omega}|u|^{4}\text{d}x\text{d}t\bigg)^{\frac{1}{3}}
 \bigg(\int_{0}^{T}\int_{\Omega}|\text{rot}A|^{2}\text{d}x\text{d}t\bigg)^{\frac{2}{3}}
\\ & \leq C\bigg(\int_{0}^{T}\int_{\Omega}|u|^{4}\text{d}x\text{d}t\bigg)^{\frac{1}{3}}
 \bigg(\int_{0}^{T}\int_{\Omega}|\nabla A|^{2}\text{d}x\text{d}t\bigg)^{\frac{2}{3}}
 \\ & <\infty,
\end{aligned}
\end{equation}
which in turn implies
\begin{equation}\label{P003}
u\times \text{rot}A\in L^{\frac{4}{3}}(0,T;L^{\frac{4}{3}}(\Omega,\mathbb{R}^2)).
\end{equation}

Recall that $(u,p)$ satisfying  the following Stokes system
\begin{eqnarray}\label{P04}
 \left\{
   \begin{aligned}
 & \frac{\partial u}{\partial t}=\nu\Delta u-\frac{1}{\rho_0}\nabla p+F(x,t),
 \\&  \nabla\cdot u=0,
 \\&  u|_{\partial \Omega}=0,
 \\&  u(0)=\phi,
\end{aligned}
\right.
\end{eqnarray}
where $F(x,t)=f-(u\cdot\nabla)u+\frac{\rho_e}{\rho_0}(u\times\text{rot}A)$.

By (\ref{P02}) and (\ref{P003}), we get $F\in L^{\frac{4}{3}}(0,T;L^{\frac{4}{3}}(\Omega,\mathbb{R}^2))$. Applying this into Lemma 2.4,
 we obtain that
\begin{equation}\label{P05}
\begin{aligned}
 & u\in L^{\frac{4}{3}}(0,T;W^{2,\frac{4}{3}}(\Omega,\mathbb{R}^2)),\ u_t\in L^{\frac{4}{3}}(0,T;L^{\frac{4}{3}}(\Omega,\mathbb{R}^2)),
 \\ & p\in L^{\frac{4}{3}}(0,T;W^{1,\frac{4}{3}}(\Omega)).
 \end{aligned}
\end{equation}

In the next step, the Lemma 2.5 will be used,  since (\ref{P04}) can be rewritten  as the following stationary Stokes equations
\begin{eqnarray}\label{P06}
 \left\{
   \begin{aligned}
 & -\nu\Delta u+\frac{1}{\rho_0}\nabla p =\widetilde{F}(x,t),
 \\&  \nabla\cdot u=0,
 \\&  u|_{\partial \Omega}=0,
 \\&  u(0)=\phi,
\end{aligned}
\right.
\end{eqnarray}
where $\widetilde{F}(x,t)=f-(u\cdot\nabla)u+\frac{\rho_e}{\rho_0}(u\times\text{rot}A)-u_t$.
\\
\textbf{Step 3 The estimate for}  $||\widetilde{F}||_ {L^\infty(\Omega,L^2(\Omega,\mathbb{R}^2))}$.

\vskip2mm
{\textbf{(i)  The estimate for}} $||\nabla u||_{L^\infty(0,T;L^2)}$.

Multiplying Eq. (\ref{Eq})$_1$ by $u_t$ and integrating over $\Omega$, we have
\begin{equation}\label{cha01}
  \frac{\mu}{2}\frac{\text{d}}{\text{d}t}\int_\Omega|\nabla u|^2\text{d}x
 +\int_\Omega|u_t|^2\text{d}x
 =\int_{\Omega}-(u\cdot \nabla)u\cdot u_t+\frac{\rho_e}{\rho_0}(u\times\text{rot}A)u_t\text{d}x.
\end{equation}
Note that  the  following continuous embeddings
\begin{equation}\label{P08}
W^{2,\frac{4}{3}}(\Omega,\mathbb{R}^2)\hookrightarrow W^{1,4}(\Omega,\mathbb{R}^2)\hookrightarrow C^{\frac{1}{2}}(\Omega,\mathbb{R}^2) \hookrightarrow C^{0}(\Omega,\mathbb{R}^2).
\end{equation}
Combining (\ref{P08}), H\"{o}lder inequality and $\epsilon-$Young inequality, we derive that
\begin{equation}\label{cha02}
 \begin{aligned}
 \int_{\Omega}|(u\cdot \nabla)u\cdot u_t|\text{d}x
 &\leq C||u_t||_{L^2}||u||_{C^0}||\nabla u||_{L^2}
\\ &\leq \frac{1}{4}||u_t||^2_{L^2}+C^2||u||^2_{C^0}||\nabla u||^2_{L^2}
 \end{aligned}
\end{equation}
and
\begin{equation}\label{cha03}
 \begin{aligned}
 \frac{\rho_e}{\rho_0}\int_{\Omega}|(u\times\text{rot}A)u_t|\text{d}x
 &\leq C||u||_{C^0}||\nabla A||_{L^2}||u_t||_{L^2}
\\ &\leq \frac{1}{4}||u_t||^2_{L^2}+C^2||u||^2_{C^0}||\nabla A||^2_{L^2},
 \end{aligned}
\end{equation}
which together with Gronwall's inequality implies
\begin{equation}\label{cha04}
   \begin{aligned}
  &\text{ess}\sup\limits_{0<t<T}||\nabla u||_{L^2}<\infty.
  \end{aligned}
\end{equation}
\vskip2mm
{\textbf{(ii)   The estimate for}} $|| u_t||_{L^\infty(0,T;L^2)}$.

Taking $t$-derivative of  Eq. (\ref{Eq})$_1$, then one gets that
\begin{equation}\label{cha1}
u_{tt}-\mu\Delta u_t=-(u_t\cdot \nabla)u-(u\cdot \nabla)u_t-\frac{1}{\rho_0}\nabla p_t+\frac{\rho_e}{\rho_0}u_t\times\text{rot}A
+\frac{\rho_e}{\rho_0}u\times\text{rot}A_t.
\end{equation}
Multiplying (\ref{cha1}) by $u_t$ and integrating over $\Omega$, we obtain
\begin{equation}\label{cha2}
  \frac{1}{2}\frac{\text{d}}{\text{d}t}\int_\Omega|u_t|^2\text{d}x
 +\mu\int_\Omega|\nabla u_t|^2\text{d}x
 =\int_{\Omega}-(u_t\cdot \nabla)u\cdot u_t+\frac{\rho_e}{\rho_0}(u\times\text{rot}A_t)u_t\text{d}x.
\end{equation}
since
\begin{equation*}
(u_t\times\text{rot}A)\cdot u_t=0, \int_\Omega(u\cdot \nabla)u_t\cdot u_t\text{d}x=-\int_\Omega\frac{1}{2}u^2_t\text{div}{u}\text{d}x=0.
\end{equation*}

Next, we estimate the two terms on the right hand of (\ref{cha2}). By (\ref{P08}) and integrating by parts yield
 \begin{equation}\label{cha3}
 \begin{aligned}
 &-\int_{\Omega}(u_t\cdot \nabla)u\cdot u_t\text{d}x=\int_{\Omega}u^i_tu^j\partial_iu^j_t
 -u^i_t\partial_i(u^ju^j_t)\text{d}x
\\ &=\int_{\Omega}u^i_tu^j\partial_iu^j_t\text{d}x\leq C||u||_{C^0}\bigg(||u_t||^2_{L^2}+||\nabla u_t||^2_{L^2}\bigg).
 \end{aligned}
\end{equation}
And similarly,
\begin{equation}\label{cha4}
 \begin{aligned}
 \frac{\rho_e}{\rho_0}\int_{\Omega}|(u\times\text{rot}A_t)u_t|\text{d}x&\leq \frac{C\rho_e}{\rho_0}\int_{\Omega}|uu_t\nabla A_t|\text{d}x
\\ &\leq \frac{C\rho_e}{\rho_0}||u||_{C^0}\bigg(||u_t||^2_{L^2}+||\nabla A_t||^2_{L^2}\bigg).
 \end{aligned}
\end{equation}
Hence, by (\ref{cha2}), (\ref{cha3}) and (\ref{cha4}), we get that
\begin{equation}\label{cha5}
   \begin{aligned}
  &\frac{1}{2}\frac{\text{d}}{\text{d}t}\int_\Omega|u_t|^2\text{d}x
 +\mu\int_\Omega|\nabla u_t|^2\text{d}x
 \\ &\leq||u||_{C^0}\bigg((1+C\rho_e/\rho_0)||u_t||^2_{L^2}+||\nabla u_t||^2_{L^2}\bigg)+\frac{C\rho_e}{\rho_0}||u||_{C^0}||\nabla A_t||^2_{L^2},
  \end{aligned}
\end{equation}
which together with Gronwall's inequality completes the estimate
\begin{equation}\label{cha6}
   \begin{aligned}
  &\text{ess}\sup\limits_{0<t<T}||u_t(t)||_{L^2}<\infty.
  \end{aligned}
\end{equation}
\vskip 2mm
{\textbf{(iii)  The estimates for}} $||(u\cdot\nabla u)||_{L^\infty(0,T;L^2)}$ and $|| u\times A||_{L^\infty(0,T;L^2)}$.

From (\ref{cha04}), it is easy to see that
\begin{equation}\label{cha61}
\nabla u\in L^\infty(0,T;Y).
\end{equation}
Hence
\begin{equation*}
 u\in L^\infty(0,T;H^1).
\end{equation*}

It is known that $H^1\hookrightarrow L^q(1<q<\infty)$ when $n=2$. Note that
\begin{equation}\label{cha62}
 \begin{aligned}
 \bigg(\int_{\Omega}|(u\cdot \nabla )u|^r\text{d}x\bigg)^\frac{1}{r}
 &\leq \bigg(\int_{\Omega}|\nabla u|^2\text{d}x\bigg)^\frac{1}{2}\bigg(\int_{\Omega}|u|^{\frac{2r}{2-r}}
 \text{d}x\bigg)^\frac{2-r}{2r}<\infty
 \end{aligned}
\end{equation}
provided that  $1<r<2$. Hence
\begin{equation}\label{cha611}
(u\cdot \nabla )u \in L^\infty(0,T;L^r(\Omega,\mathbb{R}^2)).
\end{equation}

By using the H\"{o}lder inequality and the Sobolev embedding theorem,  it follows that
\begin{equation}\label{cha63}
 \begin{aligned}
 \int_{\Omega}|u\times\text{rot}A|^2\text{d}x
 &\leq C\int_{\Omega}|u|^2|\nabla A|^2\text{d}x
 \\ & \leq C\bigg(\int_{\Omega}|u|^4\text{d}x+\int_{\Omega}|\nabla A|^4\text{d}x\bigg)
  \\ & \leq C\bigg(\int_{\Omega}|\nabla u|^2\text{d}x+\int_{\Omega}|\Delta A|^2\text{d}x\bigg).
  \end{aligned}
\end{equation}
Together  (\ref{P4}) with  (\ref{cha63}), we have
\begin{equation}\label{cha612}
u\times\text{rot}A \in L^\infty(0,T;L^2(\Omega,\mathbb{R}^2)).
\end{equation}
According to (\ref{cha6}), (\ref{cha611}), (\ref{cha612}) and the assumption, $\widetilde{F}$ in (\ref{P06}) satisfies
 \begin{equation}\label{cha64}
\widetilde{F}\in L^r(\Omega,\mathbb{R}^2)(1<r<2) \quad\text{for~any} \ 0<T<\infty.
\end{equation}
Applying (\ref{cha64}) into  Lemma 2.5, we get
\begin{equation}\label{P007}
u\in L^\infty(0,T;W^{2,r}(\Omega,\mathbb{R}^2)),\ p\in L^\infty(0,T;W^{1,r}(\Omega)).
\end{equation}

Using the Sobolev embedding theorem $W^{2,r}\hookrightarrow C^\alpha\hookrightarrow C^0(0<\alpha<1,n=2,)$, we deduce from (\ref{cha61}) and (\ref{P007}) that
\begin{equation}\label{cha65}
(u\cdot \nabla )u \in L^2(\Omega,\mathbb{R}^2)\quad \quad\text{for~any} \ 0<T<\infty.
\end{equation}

By (\ref{cha6}), (\ref{cha612}) and  (\ref{cha65}), we get that
\begin{equation}\label{cha05}
\widetilde{F}=f-(u\cdot\nabla)u+\frac{\rho_e}{\rho_0}(u\times\text{rot}A)-u_t\in L^\infty(\Omega,L^2(\Omega,\mathbb{R}^2)).
\end{equation}
Applying  (\ref{cha05}) into  Lemma 2.5,  we obtain that for any $T>0$
\begin{equation}\label{P07}
u\in L^\infty(0,T;W^{2,2}(\Omega,\mathbb{R}^2)),\ p\in L^\infty(0,T;W^{1,2}(\Omega)).
\end{equation}
Therefore, (\ref{P0}), (\ref{P4}), (\ref{P6}),  (\ref{cha6}) and (\ref{P07}) complete the proof.

 \end{proof}

\end{document}